\documentclass[a4paper]{article}
\usepackage{latexsym,amsmath,amsthm,amssymb,fullpage}
\usepackage{color}
\usepackage{ulem}
\usepackage{authblk}
\usepackage{hyperref}

\usepackage[latin1]{inputenc}
\newtheorem{theorem}{Theorem}

\newtheorem{corollary}[theorem]{Corollary}

 \newcommand{\Tau}{\mathrm{\mathcal{T}}}

\theoremstyle{remark} \newtheorem{remark}[theorem]{Remark}

\def\be{\begin{equation}}
\def\en{\end{equation}}
\def\bee{\begin{eqnarray*}}
\def\ene{\end{eqnarray*}}

\usepackage{color}

\def\R{{\mathbb{R}}}

\def\Var{{\rm Var}}

\def\Hess{{\rm Hess}}

\begin{document}   

\title {Quantitative logarithmic Sobolev inequalities \\ and stability estimates}
\author{M. Fathi\thanks{Universit\'e Pierre et Marie Curie, Paris, France, max.fathi@etu.upmc.fr.} , E. Indrei\thanks{Carnegie Mellon University, Pittsburgh, USA, egi@cmu.edu. E. Indrei was supported by US NSF Grants OISE-0967140\\ \indent \indent (PIRE), DMS-0405343, and DMS-0635983 administered by the Center for Nonlinear Analysis in Pittsburgh, PA.} , and M. Ledoux\thanks{University of Toulouse, Toulouse, France, and Institut Universitaire de France, ledoux@math.univ-toulouse.fr.
}}

\date{}
\makeatletter
\def\blfootnote{\xdef\@thefnmark{}\@footnotetext}
\makeatother
\maketitle

%\blfootnote{E. Indrei was supported by US NSF Grants OISE-0967140 (PIRE), DMS-0405343, and DMS-0635983 administered by the Center for Nonlinear Analysis at Carnegie Mellon University.} 
%

\begin{abstract}
We establish an improved form of the classical logarithmic Sobolev inequality for the Gaussian
measure restricted to probability densities which satisfy a Poincar\'e inequality.
The result implies a lower bound
on the deficit in terms of the quadratic Kantorovich-Wasserstein distance. We similarly investigate
the deficit in the Talagrand quadratic transportation cost inequality this time by means of an
${\rm L}^1$-Kantorovich-Wasserstein distance, optimal for product measures, and deduce a lower bound on the deficit in the logarithmic Sobolev inequality in terms of this metric. Applications are given in the context
of the Bakry-\'Emery theory and the coherent state transform. The proofs combine tools from
semigroup and heat kernel theory and optimal mass transportation.
\end{abstract}

\bigskip

\section{Introduction and main results}

The classical logarithmic Sobolev inequality of L.~Gross \cite {g75} for the
standard Gaussian measure
$$
 d\gamma (x) \, = \, d \gamma^n (x) \, = \,  e^{-|x|^2/2} \, \frac {dx}{(2\pi)^{n/2}}
$$ 
on the Borel sets of $\R^n$ (cf.~e.g.~\cite {v03,v09,bgl14})
states that if $d\nu = f d\gamma$ is a probability measure with
density $f$ with respect to $\gamma$,
\be \label {eq.lsi}
{\rm H}(\nu) \, \leq \, \frac {1}{2} \, {\rm I} (\nu)
\en 
where
$$
{\rm H}(\nu)  \, = \, {\rm H} \big(\nu \, | \, \gamma ) \, = \, 
\int_{\R^n} f \log f  \, d \gamma  
$$
is the relative entropy of $\nu$ with respect to $\gamma$ and
$$
{\rm I}(\nu)  \, = \, {\rm I} \big(\nu \, | \, \gamma ) \, = \, 
\int_{\R^n} \frac {|\nabla f|^2}{f} \,  d \gamma 
$$
is the Fisher information of $\nu$ with respect to $\gamma$. 

Logarithmic Sobolev inequalities (LSI) are a useful tool in analysis and probability
in the study of convergence to equilibrium, large deviations, and measure concentration.
They are also equivalent to hypercontractivity for their associated semigroup
(cf.~\cite {v03,v09,bgl14}). To ensure that the various terms of the LSI are well-defined,
some smoothness and positivity properties of the density $f$ of $\nu$ have to be considered.
These may be handled by approximation and regularization (see e.g.~\cite {bgl14}).
When dealing with entropy
${\rm H}(\nu)$ and Fisher information ${\rm I}(\nu)$
(and below the LSI deficit $\delta_{\rm LSI}(\nu) $ \eqref{eq.lsideficit}), it will be
usually implicitly
understood that they are well-defined (and finite) for suitable density functions $f$.

The constant $1/2$ in the Gaussian LSI \eqref {eq.lsi}
is known to be optimal, and it was first shown in \cite{c91b} that the cases of equality
are exactly the measures of the form
\be \label {eq.extremal}
d\gamma_b(x) \, = \,  e^{b \cdot x - \frac{|b|^2}{2}} \, d\gamma(x), \quad b \in \R^n.
\en
In other words, the extremal densities $f$ are exponential functions.
(Note that $b$ is the barycenter of $\gamma_b$,
so that in particular the only centered extremal measure is $\gamma $ itself.)

However, the study of the {\it logarithmic Sobolev deficit}
\be \label {eq.lsideficit}
\delta_{\rm LSI} (\nu) \, = \,  \frac {1}{2} \, {\rm I} (\nu) - {\rm H}(\nu)
\en
to quantify proximity with the extremal measures is still largely open in spite
of recent developments for classical Sobolev and related isoperimetric inequalities.
In the broader context of stability results for functional inequalities,
when looking at a functional inequality with known optimal constants and optimizers,
a natural question is indeed whether functions that are close to achieving the optimum are
close to some optimizer. The task is to bound from below the deficit by
some functional that measures how far we are from some optimizer
(typically, a distance). Examples of such results are the recent quantitative stability estimates for Sobolev \cite{cfmp09, fmp13}, Brunn-Minkowski \cite{fmp09, fj13}, and isoperimetric inequalities \cite{fmp08, fmp10, fi13, in14}.

\medskip

The first main result of this note is to propose a (strict) strengthening of the Gaussian LSI
\eqref {eq.lsi} within a subclass of probability measures $\nu$
which in turn produces a lower bound on the deficit $ \delta_{\rm LSI}(\nu)$.
Denote by ${\cal P}(\lambda)$ the class of probability measures
$\nu $ on the Borel sets of $\R^n$
satisfying a Poincar\'e inequality with constant $\lambda >0$
in the sense that for every smooth
$g : \R^n \to \R$ such that $\int_{\R^n} g  d\nu = 0$,
\be \label {eq.poincare}
\lambda \int_{\R^n} g^2 d\nu  \, \leq \,  \int_{\R^n} |\nabla g| ^2 d\nu.
\en
Note that under such a Poincar\'e inequality, the measure $\nu$ necessarily has
a second moment.

\begin {theorem} \label {thm.improvedlsi}
For any centered ($\int_{\R^n} x d\nu = 0$) probability measure $d\nu = f d\gamma$
in the class ${\cal P} (\lambda)$,
$$ 
{\rm H}(\nu)   \, \leq \, \frac {c(\lambda)}{2} \, {\rm I}(\nu),
$$
where
$$
c(\lambda) \, = \, \frac{1 - \lambda  + \lambda \log \lambda}{(1- \lambda )^2} \, < \, 1
   \qquad  \big ( c (1) \, = \, {\textstyle \frac {1}{2}} \big) .
$$
\end {theorem}

The constant is sharp, as can be seen when taking $\nu$ with density
$ f(x) = \sqrt {\lambda} \, e^{(1-\lambda)x^2/2}$, $\lambda > 0$,
on the line. Of course, since the constant $1/2$ in the Gaussian LSI is optimal,
such a strengthening can only be expected to hold on a subset of probability measures.

 In dimension $n = 1$, the class of probability measures satisfying a Poincar\'e inequality
\eqref {eq.poincare}
has been completely characterized. A probability measure $\nu$ with density $p$
with respect to the Lebesgue measure and median $m$ satisfies a Poincar\'e inequality
if and only if the following holds (see~\cite{bg99,bgl14}):
$$
A^+ \, = \, \underset{x \geq m}{\sup} \hspace{1mm} 
   \nu \big ([x, +\infty[ \big )\int_m^x{\frac{1}{p(t)} \, dt}   \, < \, \infty ,
$$
$$
A_- \, = \, \underset{x \leq m}{\sup} \hspace{1mm} 
    \nu \big (] -\infty, x] \big )\int_x^m{\frac{1}{p(t)} \, dt}  \, < \, \infty .
$$
Moreover, the optimal Poincar\'e constant $\lambda_{\rm opt}$ for $\nu$ satisfies
$$
 \frac{1}{2} \, \max(A^+, A^-)  \, \leq \,  \lambda_{\rm opt} \, \leq \,  4\max(A^+, A^-) .$$
In higher dimension, there is no such simple characterization,
but fairly general sufficient conditions are available. For example, if $\nu$ has a density
of the form $ e^{-V}$ with respect to the Lebesgue measure, a sufficient condition is
the existence of $a \in \, ]0, 1[$ such that $a|\nabla V|^2 - \Delta V$ is bounded from
below by some positive constant outside of some ball (see \cite{bbcg08}). A more classical
condition is the Bakry-\'Emery criterion 
\be \label{eq.bakryemery}
\Hess(V) \geq \eta \, {\rm Id} \qquad {\hbox {for some}} \, \, \, \eta >0
\en
on the potential $V$ (\cite {be85,v03,bgl14}) ensuring a Poincar\'e inequality with
constant $\lambda = \eta$.

\medskip

As an equivalent formulation of Theorem~\ref {thm.improvedlsi}, for $\nu$ centered
in ${\cal P}(\lambda)$,
\begin {equation} \label {eq.equi}
\delta_{\rm LSI} (\nu) \, \geq \, c_1 (\lambda) \,  {\rm I}(\nu)
\end {equation}
where $c_1 (\lambda ) = \frac {1}{2} \, (1 - c(\lambda))$.
The non-centered version of \eqref {eq.equi}, and thus
of Theorem~\ref {thm.improvedlsi}, reads as follows.    

\begin {corollary} \label{cor.stabilityw2centered}
For any probability measure $d\nu = f d \gamma$
in the class ${\cal P} (\lambda)$ with barycenter $b = b (\nu)$,
$$
\delta_{\rm LSI}(\nu) \, \ge \,  c_1(\lambda) 
    \int_{\R^n} \big |\nabla (\log f) - b \big |^2 d\nu.
$$
\end {corollary}

Corollary~\ref {cor.stabilityw2centered} follows by a rescaling argument involving the barycenter.
For $d \nu = f d\gamma $ with mean $b$, define 
\be \label {eq.nub}
d\nu_b(x) \, = \,  f(x + b)e ^{-\big(b\cdot x + \frac{|b|^2}{2}\big)} \, d\gamma(x).
\en
The probability measure $\nu_b$ has mean $0$ and, as is easily checked, satisfies
${\rm H}(\nu_b) = {\rm H}(\nu) - \frac{|b|^2}{2}$ and
${\rm I}(\nu_b) = {\rm I}(\nu) - |b|^2$,
so that $\delta_{\rm LSI}(\nu_b) = \delta_{\rm LSI}(\nu)$.
The conclusion then easily follows.

\medskip

Theorem~\ref {thm.improvedlsi} improves upon the recent \cite{im13} 
where stronger conditions on the Hessian of the density $f$ are considered
 (in particular parts of the class $\mathcal {P} (\lambda)$), with weaker
dependence of the constant. The work \cite{im13} actually investigates
how far an admissible density is from saturating the logarithmic Sobolev inequality
as measured with Wasserstein distance, providing a control of the deficit
$\delta_{\rm LSI} (\nu)$ in the logarithmic Sobolev inequality by
the (quadratic) Kantorovich-Wasserstein distance ${\rm W}_2(\nu, \gamma)$.
Within the class ${\cal P}(\lambda)$, this
is easily achieved via Theorem~\ref {thm.improvedlsi} together with
the Talagrand quadratic transportation cost inequality \cite {t96} (cf.~\cite {v03,v09,bgl14})
\be \label {eq.talagrand}
{\rm W}_2 (\nu, \gamma )^2 \,\leq \, 2 \, {\rm H}(\nu)
\en
holding for all probability measures $\nu $ (absolutely continuous with respect to $\gamma$).
Recall that the Kantorovich-Wasserstein distance $ {\rm W}_2 (\nu, \mu )$
between two probability measures $\nu$ and $\mu$ is given by
$$
{\rm W}_2 (\nu, \mu ) \, = \, \inf \bigg ( \int_{\R^n} \! \int_{\R^n} |x-y|^2 d \pi (x,y) \bigg)^{1/2}
$$
where the infimum is over all couplings $\pi$ of probability measures on $\R^n \times \R^n$
with respective marginals $\nu$ and $\mu$. Note that if $\nu \in {\cal P} (\lambda)$,
it has necessarily a second moment so that the
Kantorovich-Wasserstein distance ${\rm W_2} (\nu, \gamma)$ is finite.

\begin {corollary} \label{cor.stabilityw2}
For any centered probability measure $d\nu = f d \gamma$ in the class ${\cal P} (\lambda)$,
$$ 
\delta_{\rm LSI}(\nu) 
    \, \geq \,    c_2(\lambda) \, {\rm W}_2(\nu, \gamma)^2, $$
where $ c_2(\lambda)= \frac{1}{2} \big( \frac{1}{c(\lambda )} - 1 \big)$ and $c(\lambda)$
is as in Theorem \ref{thm.improvedlsi}. 
\end {corollary}

This corollary may be compared to the Otto-Villani HWI inequality \cite{ov00}
(cf.~\cite {v03,v09,bgl14}), valid for any probability $\nu$,
\be \label {eq.hwi}
{\rm H} (\nu)  \, \leq \, {\rm W}_2(\nu, \gamma) \, \sqrt {\, {\rm I}(\nu)} 
   - \frac {1}{2} \, {\rm W}_2(\nu, \gamma)^2.
\en
It should be mentioned that one cannot expect
$$ 
\delta_{\rm LSI}(\nu)   \, \geq \, c \, {\rm W}_2(\nu, \gamma)^2 
$$
to hold for some $c>0$ and all probability measures $\nu$.
Indeed, such an inequality combined with the HWI inequality would then imply the
logarithmic Sobolev inequality
$ {\rm H}(\nu) \, \leq \frac{1 + c}{2 + 4c} \, {\rm I}(\nu) $
with therefore a constant strictly better than the optimal $1/2$.
A complete stability result for the Gaussian LSI therefore requires a distance weaker than ${\rm W}_2$.

In this direction, Theorem~\ref {thm.improvedlsi} may also be used to provide a lower bound
on the deficit $ \delta_{\rm LSI}$ in terms of the total variation. Indeed, as the
standard Gaussian measure $\gamma $
satisfies a $(1,1)$-Poincar\'e inequality (cf.~e.g.~\cite {le01})
\be \label {eq.poincare11}
\int_{\R^n} |g | \, d\gamma  \, \leq \,  2 \int_{\R^n} |\nabla g| \, d\gamma
\en
for every smooth $g : \R^n \to \R$ with mean zero, if $d\nu=f d\gamma$,
$$
\int_{\R^n} |f-1| \, d\gamma 
    \, \leq \,  2 \int_{\R^n} |\nabla f| \, d\gamma   \, \leq \, 2 \, \sqrt {{\rm I}(\nu)} 
$$
by the Cauchy-Schwarz inequality. We then only state the consequence
of \eqref {eq.equi} in the centered case.

\begin{corollary} \label{cor.tv}
For any centered probability measure $d\nu = f d \gamma$ in the class ${\cal P}(\lambda)$,
$$
\delta_{\rm LSI}(\nu) \, \geq \,
    \frac{c_1(\lambda)}{4} \bigg ( \int_{\R^n} |f-1| \, d\gamma \bigg)^2
       \, = \,  \frac{c_1(\lambda)}{4} \, {|| \nu - \gamma||}^2_{\rm TV}.
$$ 
\end{corollary}

\medskip

While Corollaries~\ref {cor.stabilityw2} and \ref{cor.tv} are strictly weaker
than Theorem~\ref {thm.improvedlsi},
they have the advantage of providing a lower bound on the deficit in the Gaussian
LSI in terms of a metric. 

A one-dimensional stability result of the same kind as Corollary \ref{cor.stabilityw2}
is proven in Corollary~4.4 of \cite{bgrs13}, however with a worse constant of proportionality.
The main assumption is uniform log-concavity of $\nu$ (i.e. \eqref {eq.bakryemery}) 
which is used to apply a $(1,1)$-Poincar\'e inequality.  As far as we know,
the argument of \cite{bgrs13} does not extend to higher dimensions. Nevertheless, the one-dimensional result may be combined with a tensorization argument to cover the case of $n$-dimensional random vectors with uniformly log-concave distributions whose one-dimensional projections form a martingale. Such an assumption is not the same as simply assuming that the mean of $\nu$ is zero. More generally,  
$$ 
\delta_{\rm LSI}(\nu)
    \, \geq \, c \, {\rm W}_2(\bar{\nu}, \gamma)^2, 
$$
where $\bar{\nu}$ is the law of a random vector $\bar{X}$ obtained by
modifying a random vector $X$ with law $\nu$ in such a way that its
one-dimensional marginals $X_1,\ldots,X_n$ form a martingale \cite {bgrs13}.
For unconditional random variables, this is the same as assuming the mean to be zero, but
in general it does not seem like ${\rm W}_2(\bar{\nu}, \gamma)$ and
${\rm W}_2(\nu, \gamma)$ can be easily compared.
The contribution \cite{bgrs13} also contains deficit estimates for general $\nu$,
but with lower bounds that are either not a power of a distance,
are dimension-dependent, or involve $\bar \nu$.
For example, there is a universal constant $c>0$ such that for all smooth probability
measures $\nu$ on $\R^n$, 
\begin{equation} \label{eq.bgrs}
\delta_{\rm LSI}(\nu) \, \ge \,  c \, \frac{\Tau(\bar \nu, \gamma)^2}{{\rm H}(\bar \nu)}
\end{equation}  
where $\bar \nu$ is the previously discussed martingale rearrangement of $\nu$
and $\Tau$ is a transportation cost associated to the function $t \mapsto t - \log(1 + t)$.

\bigskip

The second main result of this note investigates the deficit in the 
Talagrand quadratic transportation cost inequality \eqref {eq.talagrand}.
A result of Otto and Villani \cite{ov00} states that a measure
satisfying a logarithmic Sobolev inequality automatically satisfies a Talagrand-type inequality.
It is easy to see, using the HWI inequality \eqref {eq.hwi}, that the cases of equality for Talagrand's
inequality are exactly the same as for the Gaussian LSI. Therefore,
it is natural to investigate lower bounds on the {\it Talagrand deficit}
$$
\delta_{\rm Tal}(\nu) \, = \, 2\, {\rm H}(\nu) - {\rm W}_2(\nu, \gamma)^2.
$$

In dimension one, it was shown by Barthe and Kolesnikov \cite{bk08}
that the deficit $ \delta_{\rm Tal}(\nu)$ satisfies
$$
\delta_{\rm Tal}(\nu) \, \geq \,
    c  \, \inf_\pi \int_{\R^n} \varphi \big (|x -y| \big ) \pi(dx, dy),
$$
where the infimum is over couplings $\pi$ of $\nu $ and $\gamma$, and
$\varphi(t) = t - \log(1 + t)$. Note that the right-hand side
in this inequality is an optimal transport cost, with a cost that is quadratic-then-linear
in the distance. This inequality immediately yields the weaker version
$$
\delta_{\rm Tal}(\nu) \, \geq \, 
    c \, \min \left({\rm W}_1(\nu, \gamma)^2, {\rm W}_1(\nu, \gamma)\right),
$$
where ${\rm W}_1$ is the ${\rm L}^1$-Kantorovitch-Wasserstein distance (with $\ell^2$-cost function on $\R^n$) between the
one-dimensional measures $\nu$ and $\gamma$.

We establish here the following multi-dimensional version of the Barthe-Kolesnikov result. Let
$$
{\rm W}_{1,1} (\nu, \mu ) 
      \, = \, \inf \int_{\R^n} \! \int_{\R^n} \sum_{i=1}^n|x_i - y_i| d \pi (x,y)
$$
be the ${\rm L}^1$-Kantorovich-Wasserstein distance with $\ell^1$-cost function on $\R^n$ where the infimum is over couplings $\pi$ of $\nu $ and $\mu$.

\begin{theorem} \label{thm.deficittal}
There is a numerical constant $c >0$ such that for any centered probability measure
$d\nu=f d\gamma$ on $\R^n$ with finite second moments and $f>0$ locally bounded,
$$
\delta_{\rm Tal}(\nu) \, \geq \, c \, \min \bigg ( \frac {{\rm W}_{1,1} (\nu, \gamma ) ^2}{n} \, ,
       \frac {{\rm W}_{1,1} (\nu, \gamma ) }{\sqrt n} \bigg )  . 
$$
\end {theorem}

One feature of this result is that it is valid for general measures. Moreover,
the lower bound is expressed in terms of a metric
on the space of probability measures on $\R^n$ and the exponent is independent
of the dimension.
In general, the deficit in Theorem \ref {thm.deficittal} is only optimal for small perturbations
of the Gaussian. For an $n$-dimensional product measure $\nu^n = \nu^{\otimes n}$,
$\delta_{\rm Tal}(\nu^n) = n \delta_{\rm Tal}(\nu)$ grows linearly in $n$. 
This is also the behavior of
$$
\frac {{\rm W}_{1,1} (\nu^n, \gamma^n ) ^2}{n} \, = \,
   n \, {\rm W}_{1,1}(\nu, \gamma^1)^2.
$$
When $ n > \! > {\rm W}_{1,1}(\nu, \gamma^1)^{-2}$,
the expected growth is lost. Nevertheless, for product measures whose
one-dimensional marginals are close enough to $\gamma = \gamma^1$
(i.e. such that ${\rm W}_{1,1}(\nu,\gamma^1)^2 \leq \frac {c}{n}$),
Theorem~\ref {thm.deficittal} yields the correct order of magnitude in the dimension.

Theorem~\ref {thm.deficittal} furthermore 
yields a new proof of the equality case for the Gaussian LSI. 
Indeed, by the HWI inequality,
$$
\delta_{\rm LSI} (\nu)  \,  \geq  \, \frac {1}{2} \Big (\sqrt{{\rm I}(\nu)} - W_2(\nu, \gamma) \Big)^2.
$$
Therefore, if  $\nu$ is such that $ \delta_{\rm LSI} (\nu) = 0$, 
then ${\rm I}(\nu) = {\rm W}_2(\nu, \gamma)^2$. By the conjunction of the Talagrand
\eqref {eq.talagrand} and LSI \eqref{eq.lsi} inequalities, 
$$
{\rm W}_2(\nu, \gamma)^2 \, \leq \, 2 \, {\rm H} (\nu) \, \leq  \, {\rm I} (\nu),
$$
so that there is also equality in Talagrand's inequality and thus $ \delta_{\rm Tal}(\nu) = 0$.
Therefore, Theorem~\ref{thm.deficittal} implies that the only centered measure
satisfying $\delta_{\rm LSI}(\nu) = 0$ is precisely $\gamma$.
The non-centered case follows as for Corollary~\ref {cor.stabilityw2centered}.

The preceding argument may be quantified in terms of the ${\rm W}_{1,1}$ metric
and yields a general stability result for LSI. Recall $\nu_b$ from \eqref {eq.nub}.

\begin{corollary} \label{cor.deficit1} 
There is a numerical constant $c >0$ such that for any probability measure $d\nu=fd\gamma$ 
on $\R^n$ with $f>0$ locally bounded and positive entropy, and with barycenter $b = b(\nu)$,
$$
\delta_{\rm LSI}(\nu) \, \geq \, \frac{c}{ {\rm H}(\nu)}
    \, \min \bigg ( \frac {{\rm W}_{1,1} (\nu_b, \gamma ) ^4}{n^2} \, ,
       \frac {{\rm W}_{1,1} (\nu_b, \gamma )^2 }{ n} \bigg ).
$$
\end{corollary}

Indeed, as above, by the HWI \eqref {eq.hwi}, logarithmic Sobolev \eqref {eq.lsi} and Talagrand's
\eqref {eq.talagrand} inequalities,
$$
\delta_{\rm LSI} (\nu)  
   \,  \geq  \, \frac {1}{2} \Big (\sqrt{{\rm I}(\nu)} - {\rm W}_2(\nu, \gamma) \Big)^2
   \,  \geq  \, \frac {1}{2} \Big (\sqrt{2 \, {\rm H}(\nu)} - {\rm W}_2(\nu, \gamma) \Big)^2.
$$
Hence
$$
\delta_{\rm LSI} (\nu)  
   \,  \geq  \,  
    \frac{(2 \, {\rm H}(\nu) - W_2^2(\nu, \gamma) )^2}{2(\sqrt{2 \, {\rm H}(\nu)} 
    + {\rm W}_2(\nu, \gamma))^2}
    \, \geq \, \frac {\delta_{\rm Tal}(\nu)^2}{16\, {\rm H}(\nu)} \, .
$$

\noindent The result then follows from Theorem~\ref {thm.deficittal} for a centered $\nu$, and
in the general case by recentering as above.

Note that the inequality given by Corollary~\ref{cor.deficit1} is of a similar
form to \eqref{eq.bgrs} established in \cite{bgrs13} for smooth measures.
It does not seem that the measure $\bar{\nu}$ involved in \eqref{eq.bgrs} is directly
comparable to $\nu$ in general, whereas $\nu_b$ is an explicit transformation of
$\nu$. In particular, Corollary~\ref{cor.deficit1} immediately implies the equality
cases of LSI for general measures without any additional argument.

\medskip

Finally, there is also a lower bound on the deficit $\delta_{\rm LSI}(\nu)$ which may be expressed
only in terms of Kantorovich-Wasserstein distances. For simplicity, only the centered case
is considered.

\begin{corollary} \label{cor.deficit2}
There is a numerical constant $c >0$ such that for any
centered probability measure $d\nu = f d\gamma $ on $\R^n$
$$
\delta_{\rm LSI}(\nu) \, \geq \,
    \min \Bigg [ \frac{c \, {\rm W}_{1,1}(\nu, \gamma)^4}{n^2 \, {\rm W}_2(\nu, \gamma)^2} \, ,
             \frac {1}{2} \Bigg (\sqrt{{\rm W}_2(\nu, \gamma)^2 
         + \frac {c \, {\rm W}_{1,1}(\nu, \gamma)}{\sqrt n}} -  {\rm W}_2(\nu, \gamma)
         \Bigg)^2 \Bigg ].
$$
\end{corollary}

For the proof, argue as for Corollary~\ref {cor.deficit1} combining the HWI, logarithmic Sobolev and
Talagrand inequalities to get that
$$
\delta_{\rm LSI}(\nu) \, \geq \,
   \frac {1}{2} \Big ( \sqrt { {\rm W}_2(\nu, \gamma)^2 + \delta_{\rm Tal}(\mu)}
           -  {\rm W}_2(\nu, \gamma) \Big )^2.
$$
Write $ {\rm W}_2 = {\rm W}_2(\nu, \gamma)$ and
$ {\rm W}_{1,1} = {\rm W}_{1,1}(\nu, \gamma)$ to ease the notation.
By Theorem~\ref {thm.deficittal},
\begin {equation*} \begin {split}
\delta_{\rm LSI}(\nu) 
   & \, \geq \, \frac {1}{2} \, 
      \min \Bigg [ \Bigg (\sqrt{{\rm W}_2^2 
         + \frac {c \, {\rm W}_{1,1}^2}{n}} - {\rm W}_2 \Bigg)^2,
         \Bigg (\sqrt{{\rm W}_2^2 
         + \frac {c \, {\rm W}_{1,1}}{\sqrt n}} -  {\rm W}_2 \Bigg)^2 \Bigg ] \\
& \, = \, \frac {1}{2} \, 
   \min \Bigg [ {\rm W}_2^2
   \Bigg ( \sqrt { 1 + \frac {c \, {\rm W}_{1,1}^2}{n \, {\rm W}_2} } - 1 \Bigg )^2,
         \Bigg (\sqrt{{\rm W}_2^2 
         + \frac {c \, {\rm W}_{1,1}}{\sqrt n}} -  {\rm W}_2 \Bigg)^2 \Bigg ]. \\
\end {split} \end {equation*}
Since $ {\rm W}_{1,1}^2 \leq n \, {\rm W}_2^2$,
$$
\sqrt { 1 + \frac {c \, {\rm W}_{1,1}^2}{n \, {\rm W}_2^2} }  \, \geq \, 
   1 + \frac {c' \, {\rm W}_{1,1}^2}{n \, {\rm W}_2^2} 
$$
for some $c'>0$ only depending on $c $, and the claim follows.

\medskip

The rest of the paper is organized as follows. In Section~\ref{S2}, we prove the main results.
In Section~\ref{S3}, we establish several one-dimensional results.
Lastly, in Section~\ref{S4}, we present an improvement of the Bakry-\'Emery
theorem for symmetric measures satisfying a Poincar\'e inequality and obtain
quantitative versions of the Wehrl conjectures established by Lieb \cite{li78}
and Carlen \cite{c91} in the context of the coherent state transform.

\section{Proofs of Theorems~\ref {thm.improvedlsi} and~\ref {thm.deficittal}} \label{S2}

We start with the proof of Theorem~\ref {thm.improvedlsi}. The results in \cite{im13}
rely on mass transportation tools. The arguments here
are based on the standard semigroup
interpolation along the Ornstein-Uhlenbeck semigroup going back
to \cite{be85} (cf.~\cite{b94,bgl14}), together with
heat kernel inequalities as developed in~\cite {bgl14} (to which we refer for the necessary
background).

\begin {proof} [Proof of Theorem \ref{thm.improvedlsi}]
Recall the Ornstein-Uhlenbeck semigroup ${(P_t)}_{t \geq 0}$ given on suitable
functions $ {g : \R^n \to \R}$ by
$$
P_t g(x) \, = \,  \int_{\R^n} g \big (e^{-t}x + \sqrt {1 - e^{-2t}} \, y \big) d \gamma (y),
  \quad t \geq 0, \, \, x \in \R^n.
$$
The Ornstein-Uhlenbeck semigroup ${(P_t)}_{t \geq 0}$
is invariant and symmetric with respect to $\gamma$
and, on smooth functions, $\nabla P_t g = e^{-t} P_t (\nabla g) $ (as vectors).
For each $t \geq 0$, set $ d\nu_t = P_t f d\gamma $. The classical de Brujin's formula
indicates that
\be \label {eq.brujin}
{\rm H}(\nu ) \, = \,  \int_0^\infty {\rm I}(\nu_t) dt.
\en
This identity follows from the fact that the Fisher information ${\rm I}(\nu_t)$ is
the time-derivative of the entropy along the Ornstein-Uhlenbeck flow.

In the first step of the argument, we show that for any
$t \geq 0$, $\nu_t$ satisfies a Poincar\'e inequality \eqref {eq.poincare} with constant
$$ 
\lambda _t \, = \,  \frac{1}{\lambda ^{-1}e^{-2t} + 1 - e^{-2t}} \, . 
$$
To prove this, consider a smooth function $g $ with
$$
\int_{\R^n} g \, d\nu_t \, = \, 
 \int_{\R^n} g \,  P_t f \, d\gamma \, = \,  \int_{\R^n} P_t g  \, d\nu \, = \, 0
$$
(by symmetry of $P_t$).
First, by the local Poincar\'e inequalities for ${(P_t)}_{t \geq 0}$ (cf.~\cite{bgl14}),
for every $ t \geq 0$,
$$ 
P_t(g ^2) \, \leq \, (P_t g)^2 + (1 - e^{-2t}) P_t \big (|\nabla g |^2 \big ) .
$$
Hence, 
$$
\int_{\R^n} g ^2 d\nu_t \, = \,   \int_{\R^n} P_t (g ^2) d\nu 
\, \leq \,  \int_{\R^n} (P_t g )^2 d\nu 
     + (1 - e^{-2t}) \int_{\R^n} P_t \big (|\nabla g |^2 \big ) d\nu .
$$
Then, by the Poincar\'e inequality applied to $P_t g $, since $\int_{\R^n} P_t g d\nu = 0$,
\begin {equation*} \begin {split}
\int_{\R^n} g ^2 d\nu_t 
& \, \leq \,  \frac{1}{\lambda } \int_{\R^n} |\nabla P_t g |^2 d\nu
               + (1 - e^{-2t}) \int_{\R^n} P_t \big (|\nabla g |^2 \big )  d\nu \\\
&  \, \leq \,  \Big ( \frac{e^{-2t}}{\lambda } + 1 - e^{-2t} \Big)
     \int_{\R^n}  P_t \big ( |\nabla g |^2 \big ) d\nu \\
               &  \, \leq \,  \Big ( \frac{e^{-2t}}{\lambda } + 1 - e^{-2t} \Big)
     \int_{\R^n}   |\nabla g |^2  d\nu_t \\
\end {split} \end {equation*}
where we used the heat kernel inequality $ |\nabla P_t g |^2 \leq e^{-2t} P_t (|\nabla g |^2)$
and again the symmetry of $P_t$. The claim follows.

Towards the second step of the argument, recall that by integration by parts, for every $t >0$,
$$
{\rm I}(\nu_t) \,=\,  \int_{\R^n} \frac{|\nabla P_t f |^2}{P_t f} \, d \gamma 
     \, = \,  \int_{\R^n} P_t f |\nabla \log P_t f|^2 d\gamma  
     \, = \, \int_{\R^n} |\nabla \log P_t f|^2 d\nu_t .
$$
As is classical (cf.~\cite{b94,bgl14}),
\be \label {eq.derivativefisher}
\frac{d }{dt} \, {\rm I}(\nu_t) 
   \, = \,  - 2 \int_{\R^n} P_t f \, \Gamma_2( \log P_t f ) d\gamma  
   \, = \,  - 2 \int_{\R^n} \Gamma _2( \log P_t f) d\nu_t 
\en
where $ \Gamma_2(v) = |\Hess (v) |^2 + |\nabla v|^2$.

Since $\nu $ has a first moment, $|\nabla P_t f| \in {\rm L}^1(\gamma)$
for every $t >0$. Then, if $v_t = \log P_t f$, by the Gaussian integration by parts formula,
$$
\int_{\R^n} \! \nabla v_t \, d\nu_t \, = \, \int_{\R^n} \! \nabla P_t f \, d \gamma
   \, = \, \int_{\R^n} x  P_t f \, d \gamma.
$$
By symmetry,
$$
\int_{\R^n} x  P_t f \, d \gamma \, = \, \int_{\R^n} P_t x \, f \, d \gamma
     \, = \, e^{-t} \int_{\R^n} x f \, d \gamma = 0.
$$ 
Since $\nu_t$ satisfies a Poincar\'e inequality with constant $\lambda _t$,
applied to $v_t = \log P_t f$ for which therefore $\int_{\R^n} \nabla v_t d\nu_t = 0$,
$$
\lambda _t \int_{\R^n} |\nabla v_t|^2 d\nu_t 
   \,\leq \, \int_{\R^n} \big | \Hess (v_t) \big |^2 d \nu_t . 
$$
As a consequence,
$$ 
\frac{d}{dt} \, {\rm I}(\nu_t)  \,\leq \,  - 2 (\lambda_t + 1) \,  {\rm I}(\nu_t) . 
$$
Integrating this differential inequality, for every $ t \geq 0$,
$$ 
{\rm I}(\nu_t)  \, \leq \,  {\rm I}(\nu) \, e^{-4t} \, \frac{\lambda _t}{\lambda } \,  .
$$
Finally, by de Brujin's formula \eqref {eq.brujin}, the conclusion follows.
The proof of Theorem~\ref {thm.improvedlsi} is complete.
\end {proof}

\noindent We now turn to the proof of Theorem~\ref {thm.deficittal}, which is based
on mass transportation arguments.

\begin{proof}[Proof of Theorem \ref{thm.deficittal}]
The starting point is Cordero-Erausquin's mass transportation
proof of Talagrand's inequality \cite{ce02}.
Let $d\nu = f d\gamma$ be centered and 
$$
T \, = \,  (T_1, \ldots, T_n) : \R^n \, \to \,  \R^n
$$
be the Brenier map pushing $\gamma $ onto $\nu $.
It satisfies the Monge-Amp\`ere equation 
$$
e^{-|x|^2/2} \, = \,  f\big (T(x) \big ) e^{ -|T(x)|^2/2}  \det \big (\nabla T(x) \big ),
$$
 $d\gamma$-a.e. in the sense of Alexandrov \cite{M95, caf92}. 
Following \cite{ce02},
$$
{\rm H} (\nu) \, \geq \,
    \frac{1}{2}\, {\rm W}_2(\nu, \gamma)^2 
       + \int_{\R^n} \big [ \Delta \theta - \log \det \big ({\rm Id} + \Hess (\theta) \big ) \big] d\gamma
$$
where $\nabla \theta (x) = T(x) - x$.
Since the Laplacian is the sum of the eigenvalues of the Hessian, 
and since by the Brenier theorem $T$ is given by the gradient $\nabla \phi$
of a convex function $\phi : \R^n \to \R$
(cf.~\cite {v03,v09}), denoting by $\lambda_1,\ldots, \lambda_n$
the non-negative eigenvalues of $\nabla T$, we have
$$
\delta_{\rm Tal}(\nu)
  \, \geq \, \int_{\R^n} \sum_{i=1}^n   [ \lambda_i - 1 - \log \lambda_i  ] \,  d\gamma
  \, \geq \, \frac {1}{6} \int_{\R^n} \sum_{i=1}^n  \min \big ( | \lambda_i - 1|^2,
      |\lambda_i - 1| \big ) d\gamma.
$$
Let $I  = \{ 1 \leq i \leq n \, ;  |\lambda_i - 1 | \leq 1\}$. Then
\begin {equation*} \begin {split}
\sum_{i=1}^n  \min \big ( | \lambda_i - 1|^2,
      |\lambda_i - 1| \big )
   & \, = \,  \sum_{i \in I} |\lambda_i  - 1 |^2 
         +  \sum_{i \in I^c} |\lambda_i  - 1 | \\
      & \, \geq \,  \sum_{i \in I} |\lambda_i  - 1 |^2 
         +   \sqrt {\sum_{i \in I^c} |\lambda_i  - 1 |^2 }. \\
\end {split} \end {equation*}
Hence
\begin {equation*} \begin {split}
\delta_{\rm Tal} (\nu) 
     & \, \geq \, \frac {1}{6} \int \sum_{i \in I}  |\lambda_i  - 1 |^2  d\gamma 
        + \frac {1}{6} \int \sqrt { \sum_{i \in I^c}  |\lambda_i  - 1 |^2 } \, d\gamma  \\
          & \, \geq \, \frac {1}{6} \Bigg (\int \sqrt {\sum_{i \in I}  |\lambda_i  - 1 |^2 }
             \,  d \gamma \Bigg)^2
        + \frac {1}{6} \int \sqrt {\sum_{i \in I^c}  |\lambda_i - 1 |^2 } \, d\gamma  \\
\end {split} \end {equation*}
by Jensen's inequality. Assuming that $ \delta_{\rm Tal} (\nu)  \leq \alpha$ for some $\alpha >0$, 
$$
\delta_{\rm Tal} (\nu)  \, \geq \, 
  \frac {1}{6} \Bigg ( \int \sqrt { \sum_{i \in I}  |\lambda_i  - 1 |^2 }
              \,   d\gamma  \Bigg)^2
        + \frac {1}{36 \alpha } 
          \Bigg ( \int \sqrt { \sum_{i \in I^c}  |\lambda_i  - 1 |^2 } \, d\gamma  \Bigg )^2 .
$$
Then
\begin {equation*} \begin {split}
\delta_{\rm Tal} (\nu) 
          & \, \geq \, \frac {1}{72 \max (\alpha, 1)} \Bigg ( \int \sqrt {\sum_{i \in I}  |\lambda_i (x) - 1 |^2 }
                \, d\gamma 
   +  \int \sqrt {\sum_{i \in I^c}  |\lambda_i  - 1 |^2 } \,  d\gamma  \Bigg )^2 \\
    & \, \geq \, \frac {1}{72 \max (\alpha , 1)} \Bigg ( \int \sqrt { \sum_{i=1}^n  |\lambda_i  - 1 |^2 }
                \, d\gamma  \Bigg )^2 . \\
\end {split} \end {equation*}
Now, by the Cauchy-Schwarz inequality,
$$
\sqrt { \sum_{i=1}^n  |\lambda_i  - 1 |^2} \, = \,
   \sqrt { \sum_{i,j=1}^n \big |{(\nabla T)}_{ij} - \delta_{ij} \big |^2 }
        \, \geq \, \frac {1}{\sqrt {n}} \sum_{i=1}^n \big | \nabla (T_i - x_i) \big | .
$$ 
The characterization $T=\nabla \phi$, where $\phi:\R^n \to \R$ is convex,
implies that $\phi$ is an Alexandrov solution to 
$$
\det \big ({\rm Hess} (\phi) \big ) \, = \,  \frac{e^{-|x|^2/2}}{f(T(x)) e^{ -|T(x)|^2/2}} \, .
$$
Since $f>0$ and $T$ are locally bounded, the right-hand side is bounded away
from zero and infinity on every compact set. In particular, $\phi$ is $W^{2,1}$ \cite{M13}
(see also Remark~\ref{rmk.caf} below). 
The $(1,1)$-Poincar\'e inequality \eqref {eq.poincare11} holds for mean zero
$W^{1,1}$ functions. Observing that $\int_{\R^n} [T_i(x) - x_i ] d\gamma = 0$, $i = 1, \ldots, n$,
we thus obtain that
$$
\delta_{\rm Tal}(\nu)  \, \geq \, \frac{1}{288n \max (\alpha, 1)}\bigg (\int_{\R^n}
     \sum_{i=1}^n  | T_i - x_i | d \gamma \bigg )^2 \, \geq  \, 
        \frac{1}{288n \max (\alpha , 1)}\, {\rm W}_{1,1} (\nu, \gamma)^2.
$$
As a result, for every $\alpha >0$,
$$
\delta_{\rm Tal}(\nu)  \, \geq \, \min \bigg (
        \frac{{\rm W}_{1,1} (\nu, \gamma)^2}{288n \max (\alpha , 1)} \, , \alpha \bigg).
$$
Optimizing in $\alpha >0$ concludes the proof of Theorem~\ref{thm.deficittal}.
\end{proof}

\begin{remark} \label{rmk.caf}
In the proof of Theorem~\ref {thm.deficittal}, \cite{M13} was employed to deduce
$ W^{2,1}$-regularity of the potential function $\phi$. In our framework,
one may also infer the regularity in a different way. Indeed, from \cite{caf90} it follows that if $\phi$ is not strictly convex at a point, then it is affine on a line.
Since $\phi$ is globally convex, this implies that it only depends on $(n-1)$ variables.
In particular, $\nabla \phi(\mathbb{R}^n)$ is contained in an $(n-1)$-dimensional subspace,
and this contradicts that $\nabla \phi$ pushes $d\gamma$ onto $fd\gamma$.
Hence, $\phi$ is strictly convex on $\R^n$, and the desired regularity follows from \cite{df13}.
\end{remark}

\section{One dimensional estimates via mass transfer} \label{S3}

The proof of Theorem~\ref {thm.improvedlsi} relies on heat kernel theory. 
In this section, we establish an ${\rm L}^1$ estimate via mass transfer theory
for measures satisfying a $(1,1)$-Poincar\'e inequality on the real line
\be \label {eq.poincare11nu}
\lambda \int_{\R} |g | \, d\nu  \, \leq \,  \int_{\R} |\nabla g| \, d\nu
\en
for some $\lambda >0$ and every smooth mean zero $g : \R \to \R$.
Sufficient conditions to guarantee the (1,1)-Poincar\'e are given in \cite{bbcg08}
(see e.g.~Theorem 1.5 there). 
In general, the ${\rm L}^1$ Poincar\'e is stronger than the standard ${\rm L}^2$ inequality
\eqref {eq.poincare}, which makes Theorem \ref{thm.poincare11} below weaker than
Theorem~\ref {thm.improvedlsi}. However, the emphasis here is on the method of proof.  

\begin{theorem} \label{thm.poincare11} 
Let $d\nu = f d\gamma $ be a probability measure on $\R$ with barycenter $b = b(\nu)$
satisfying a $(1,1)$-Poincar\'e inequality with constant $\lambda >0$. Then there exists
${\tilde c} = {\tilde c}(\lambda)>0$ such that if $\delta_{\rm LSI} (\nu) \le 1$, 
$$ 
\delta_{\rm LSI}(\nu ) \, \geq \,  {\tilde c} \, \bigg ( \int_\R \big |(\log f)'-b \big | d\nu \bigg)^2 .
$$
\end{theorem}

\begin{proof}
Let $T$ be the optimal transport map between $d\nu = fd\gamma$ and
$d\gamma$. Note that $T=G^{-1} \circ F$, where $F$ and $G$ are the
cumulative distribution functions of $d\nu$ and $d\gamma$, respectively.
In particular,
$$
T'(x) \, = \,  \frac{f(x)e^{-|x|^2/2}}{e^{-|T(x)|^2/2}} \, .
$$  
From Cordero-Erausquin's mass transportation proof of the logarithmic Sobolev inequality
\cite{ce02}, we extract the estimates
\begin{equation} \label{eq.cordero1}
 2\delta_{\rm LSI}(\nu) \, \geq \, \int_\R \big | T-x+(\log f)' \big|^2 d\nu (x) 
 \end{equation}
and
\begin{equation} \label{eq.cordero2}
\delta_{\rm LSI}(\nu) \, \geq \, 
    \int_\R \big [ T'-1 - \log \big (1+(T'-1) \big ) \big ] d\nu(x)  
\end{equation}
where $T$ is the optimal transport map between $d\nu = fd\gamma$ and $d\gamma$. Recall $\varphi : (-1, \infty) \to \R$ defined by $\varphi (t) = t - \log (1+ t)$ and set
\be  \label {eq.tildephi} 
   \widetilde {\varphi} (t) \, = \,  \left\{
     \begin{array}{lr}
      \frac{t^2}{6} \, , & \hskip .05in -1\leq t \leq 1,\\
       \varphi (t) - \frac{5}{6} + \log 2, \hskip .05in & \hskip 0.05in t \geq 1.
     \end{array}
   \right.
\en
Note that $\widetilde {\varphi}(t)=  \widetilde {\varphi}(|t|)$ is convex and
$\varphi (t)\geq  \frac{1}{10}\widetilde {\varphi}(t)$. By \eqref {eq.cordero2},
Jensen's inequality and the fact that $T'\ge 0$, we obtain
\be\label {eq.psi}
\delta_{\rm LSI}(\nu) \, \geq  \, \frac{1}{10}\int_\R \widetilde {\varphi} \big (|T'-1| \big ) d\nu 
  \, \geq \, \frac{1}{10}\widetilde {\varphi}\bigg (\int_\R  |T'-1| \, d\nu \bigg). 
\en
Since it is asumed that $\delta_{\rm LSI}(\nu) \leq 1$, it follows from
the properties of $\widetilde {\varphi}$ that
$$ 
\widetilde {\varphi} \bigg(\int_\R |T'-1| d\nu \bigg) \,  \geq \,  c\bigg(\int_\R |T'-1| \, d\nu \bigg)^2
$$
for a universal $c>0$. Hence
\be \label {eq.psicsq}
\delta_{\rm LSI}(\nu)   \,  \geq \,  c\bigg(\int_\R |T'-1| \, d\nu \bigg)^2 .
\en 
By the push-forward condition, 
$\int_\R (x-T) d\nu = b-\int _\R T d\nu = b$.
Thus, combining this information with \eqref{eq.cordero1}, the Cauchy-Schwarz inequality
and the $(1,1)$-Poincar\'e inequality \eqref {eq.poincare11nu},
\begin {equation*} \begin {split}
\int_\R \big |(\log f)' - b \big | d\nu
  & \, \le \,  \int_\R \big  |(\log f)'-(x-T) \big | d\nu  + \int_\R \big |(x-T)-b \big | d\nu \\
  & \, \le \,  \sqrt {2\delta_{\rm LSI} (\nu)} + \frac {1}{\lambda} \int_\R |T' - 1| \, d\nu . \\
\end {split} \end {equation*}
Together with \eqref {eq.psicsq}, the claim is easily completed.
\end{proof}

The next corollary is achieved as Corollary~\ref {cor.stabilityw2centered}.

\begin{corollary} \label{cor.poincare11}
Let $d\nu = f d\gamma $ be a centered probability measure on $\R$
satisfying a $(1,1)$-Poincar\'e inequality with constant $\lambda >0$. Then there exists
${\tilde c} = {\tilde c}(\lambda)>0$ such that if $\delta_{\rm LSI} (\nu) \le 1$, 
$$
\delta_{\rm LSI}(\nu) \, \ge \, {\tilde c} \, {||\nu - \gamma||}_{\rm TV}^2.
$$
\end{corollary}

As already mentioned, since Theorem~\ref{thm.improvedlsi} cannot hold for all probability measures,
one may not hope to generalize Corollary~\ref {cor.stabilityw2centered}
by enlarging the function space. However, this does not prevent the weaker estimates in
Theorem~\ref{thm.poincare11} and Corollary~\ref{cor.poincare11} from being true in general.
If these estimates held in full generality, without the assumption that $\nu$ satisfies
some Poincar\'e inequality, then they would automatically recover the equality
cases of the Gaussian logarithmic Sobolev inequality.

\medskip

We conclude this section by proving a version of
Corollaries~\ref {cor.stabilityw2centered} and \ref {cor.stabilityw2}
on the real line for probability measures satisfying a second moment bound
(without assuming a Poincar\'e inequality). The proof is again based on mass transfer. 
Recall the function $\tilde \varphi $ \eqref {eq.tildephi}
from the proof of Theorem~\ref {thm.poincare11}.

\begin{theorem}
Let $d\nu = f d\gamma $ be a probability measure on $\R$ with barycenter $b = b(\nu) $ such that
${\Var_\nu(x) \leq 1}$. Then, for some $C>0$,
$$
\delta_{\rm LSI} (\nu)
  \,  \geq \,  \widetilde {\varphi} \bigg ( C\int_\R \big |(\log f)' - b \big |^2 d\nu \bigg).
$$
In particular, for some numerical $c >0$,
$$
\delta_{\rm LSI} (\nu) \, \geq \, c \, {\rm W}_2(\nu, \gamma_b)^4
$$
where $\gamma_b$ is given in \eqref {eq.extremal}.
\end{theorem}

A multidimensional version of this result was proved in
\cite{bgrs13},  with a smoothness assumption on $f$. The proof there is based on a rescaling property of the LSI. 
The contribution here is an alternative technique of proof. It would be of
interest to see if the multidimensional version can be similarly obtained using transport arguments.

\begin{proof}
By approximation, it may be assumed that $f$ has compact support and
is smooth enough with derivative at least in ${\rm L}^1(\gamma)$.
Letting as above $T : \R \to \R$ be the increasing map pushing $\nu$ onto $\gamma$, we have
\begin {equation*} \begin {split}
\int_\R \big |(\log f)' - b \big |^2 d \nu
  & \,  = \, \int_\R \big |(\log f)' +(T - x) - (T - x) - b \big |^2d\nu \\
  & \, = \,  \int_\R \big |(\log f)' +(T - x) \big |^2 d\nu + b^2 - \int_\R |T- x|^2 d\nu  
    - 2\int_\R  (T - x +b) (\log f)' d\nu  \\
  & \, = \,  \int_\R \big |(\log f)' +(T - x)|^2 d\nu  +b^2 - \int_\R  |T - x|^2 d\nu   
- 2\int_\R (T - x)f' d\gamma - 2b\int_\R f'd\gamma . \\
\end {split} \end {equation*}
By Gaussian integration by parts, $\int_\R f'd\gamma = \int_\R x fd\gamma = b$ and 
similarly
$$
\int_\R (T - x)f'd\gamma \, = \,    \int_\R x(T - x) d\nu - \int_\R (T' - 1) d\nu .
$$
After some algebra, it follows that
$$
\int_\R \big |(\log f)' -  b \big |^2 d \nu 
   \, = \, \int_\R \big |(\log f)' +(T - x) \big |^2 d\nu + 2\int_\R(T' - 1) d\nu  + \Var_\nu (x) - 1. 
$$
Using \eqref{eq.cordero1} and \eqref {eq.psi}, we get that
$$
\int_\R \big |(\log f)' -  b \big |^2 d \nu 
  \, \leq \,  2\delta_{\rm LSI}(\nu) + 2\widetilde {\varphi}^{-1}\big (\delta_{\rm LSI}(\nu) \big )
   + \Var_\nu (x) - 1,
$$
where $\widetilde {\varphi}^{-1}$ is the inverse of $\widetilde \varphi$ on $\mathbb{R}^+$. Since $\widetilde {\varphi}^{-1}(x) \geq Cx$ for some $C > 0$, 
$$
\delta_{\rm LSI} (\nu)
  \,  \geq \,  \widetilde {\varphi} \bigg ( C\int_\R \big |(\log f)' - b \big |^2 d\nu \bigg).
$$
But $\int_\R |(\log f)' - b|^2 d\nu $ is the relative Fisher information of $\nu$ with
respect to the non-centered Gaussian $d\gamma_b = e^{b.x - b^2/2}d \gamma$
which satisfies a logarithmic Sobolev inequality with constant $\frac{1}{2}$. Therefore,
together with Talagrand's inequality \eqref {eq.talagrand},
$$
\int_\R \big |(\log f)' - b \big |^2 d\nu  \, \geq \,
       {\rm H} \big ( \nu \, | \, \gamma_b \big ) \, \geq \,  {\rm W}_2(\nu, \gamma_b)^2
$$
and hence
$$
\delta_{\rm LSI} (\nu) \, \geq \,  \widetilde {\varphi} \big (C \, {\rm W}_2(\nu, \gamma_b)^2 \big).
$$
By definition of the Wassertein distance ${\rm W}_2$,
$$
{\rm W}_2(\nu, \gamma_b)^2 \, \leq \, 2 \, \Var_\nu (x) + 2 \, \Var _{\gamma_b} (x) 
    \, \leq \, 4
$$
under the assumption$\Var_\nu (x) \leq 1$. 
Since $\widetilde {\varphi}$ behaves quadratically near the origin, it finally follows that
for some numerical $c >0$,
$$
\delta_{\rm LSI} (\nu) \, \geq \,  c \, {\rm W}_2(\nu, \gamma_b)^4.
$$
\end{proof}

\section{Applications} \label{S4}

\subsection{The Bakry-\'Emery theorem for symmetric measures in $\mathcal{P}(\lambda)$}
In what follows we describe an extension
of Theorem~\ref {thm.improvedlsi} to families of log-concave measures.
Let $d \mu = e^{-V} dx$ where $V : \R^n \to \R$ a smooth potential be a probability measure on $\R^n$
satisfying the convexity condition \eqref {eq.bakryemery}, that is
$\Hess ( V) \geq \eta \, {\rm Id}$ for some $\eta > 0$. The Gaussian case corresponds to the
quadratic potential $V (x) = \frac {|x|^2}{2}$ with $\eta = 1$.

Given a probability measure $ d\nu = f d\mu$ with density $f$ with respect to $\mu$,
the relative entropy and Fisher information with respect to $\mu$ are defined as in the Gaussian case by
$$
{\rm H} \big ( \nu \, | \, \mu \big ) \, = \, \int_{\R^n} f \log f \, d\mu
\qquad {\hbox {and}} \qquad 
 {\rm I} \big (\nu \, | \mu \big) \, = \,  \int_{\R^n} \frac{|\nabla f|^2}{f}\, d\mu,
$$
and the Bakry-\'Emery LSI (see \cite{be85,v03,v09,bgl14})
ensures that
$$
{\rm H} \big ( \nu \, | \, \mu \big )  \, \leq \,  \frac{1}{2\eta} \, {\rm I} \big ( \nu \, | \, \mu \big ).
$$
As for the Gaussian LSI, the proof relies on the 
semigroup ${(P_t^V)}_{t \geq 0}$ with infinitesimal generator
$\mathcal{L}^V = \Delta - \nabla V \cdot \nabla $ for which the analogues of
\eqref {eq.brujin} and \eqref {eq.derivativefisher} read, with $d\nu_t = P_t^V f d\mu$,
$$
{\rm H} \big ( \nu \, | \, \mu \big )  \,=\, \int_0^{\infty} {\rm I} \big ( \nu_t \, | \, \mu \big )  dt
$$
and 
$$
\frac{d}{dt} \, {\rm I} \big ( \nu_t \, | \, \mu \big )  \,= \, -2 \int_{\R^n} \Gamma_2(P^V_t \log f)d\nu_t
$$
where, this time, 
$$
\Gamma_2(v) \, = \, \big |\Hess  (v) \big |^2 + \langle \Hess (V) \nabla v, \nabla v \rangle
             \, \geq \,  \big |\Hess  (v) \big |^2 + \eta \, |\nabla v |^2.
$$

If we try to mimic the proof of Theorem  \ref{thm.improvedlsi} in this context,
it should be proved that
as soon as $\nu$ belongs to $\mathcal{P}(\lambda)$, $\nu_t$ belongs to $\mathcal{P}(\lambda_t)$
with
$$
\lambda_t \, = \,  \frac{1}{\lambda^{-1}e^{-2\eta t} + \eta^{-1}(1 - e^{-2\eta t})}
$$
(which is proved as in the Gaussian case), and that, whenever $\nu$ is centered,
$\int_{\R^n}\nabla v_t d\nu_t = 0$ for all $t \geq 0$ where $v_t = \log P^V_t f$.
The latter requirement is however not true in this general context. It can
nevertheless hold in some more restricted setting,
for example as soon as $V$ is even and $\nu$ is symmetric
(i.e. if $f$ is also even) in which case $\int_{\R^n} \nabla v_t d\nu_t = \int_{\R^n}\nabla V d\nu_t = 0$.

These observations lead to the following improvement of the Bakry-\'Emery theorem
for symmetric measures in $\mathcal{P}(\lambda)$. 

\begin{theorem} \label {thm.improvedbe}
Assume that $d\mu = e^{-V} dx$ is a symmetric probability measure 
such that $\Hess ( V) \geq \eta \, {\rm Id}$ for some $\eta > 0$,
and let $d\nu = f d\mu$ be a symmetric probability measure in the class $\mathcal{P}(\lambda)$
for some $\lambda > 0$. Then, for every $ t \geq 0$,
$$
{\rm I} \big ( \nu_t \, | \, \mu \big ) \, \leq \,  e^{-4\eta t} \, \frac{\lambda_t}{\lambda}
   \, {\rm I} \big ( \nu \, | \mu).
$$
Consequently, if $\lambda \neq \eta$,
$$
{\rm H} \big ( \nu \, | \mu) \, \leq \,
     \frac{\eta - \lambda -  \lambda(\ln \eta - \ln \lambda)}{2(\eta - \lambda)^2}
       \, {\rm I} \big ( \nu \, | \mu)
$$
and, if $\lambda = \eta$,
$$
{\rm H} \big ( \nu \, | \mu) \, \leq \,  \frac{1}{4\eta} \, {\rm I} \big ( \nu \, | \mu) . 
$$
\end{theorem}

Note that this result is not a stability result, since the constant given
by the Bakry-\'Emery theorem is not optimal in general.
Theorem~\ref {thm.improvedbe} nevertheless yields improved estimates on the speed of convergence
to equilibrium for the semigroup, of interest for example in the context of Monte Carlo Markov Chain 
sampling of the measure $\mu$.

Similar estimates can obtained for measures which are given by bounded perturbations of uniformly convex potentials, using the Holley-Stroock approach. 
This includes the important example of the quartic double-well potential
$V(x) = (x^2 - 1)^2$ (which is used in statistical physics for continuous versions of the Ising model).

\subsection{Coherent state transform}

For $h>0$, let $d\mu_h$ denote $h^{-n}$ times the
Lebesgue measure on $\mathbb{C}^n$ viewed as $\mathbb{R}^{2n}$.
The coherent state transform is an integral transform mapping $({\rm L}^2(\R^n), dx)$
isometrically onto a subspace of $(\R^{2n}, d\mu_h)$ and given explicitly by  
$$
\psi \, \mapsto \,  \mathcal{L} \psi(p,q) \, = \,
     e^{ip \cdot q/2h^*} \int_{\R^n} e^{i p \cdot x/h^*} e^{-|x-y|^2/2h^*} \psi(x) dx
$$
with $h^*= \frac {h}{2\pi}$. The map
$\mathcal{L}$ is built out of Weyl's representation of the Heisenberg
group and has applications in quantum mechanics, where $|\mathcal{L} \psi|^2$
is interpreted as the phase space density in the state $\psi$.
Bounds on $|\mathcal{L} \psi|^2$ are useful in estimating, e.g., the ground state energy
of a Schr\"odinger operator (see~\cite{li81,c91}).

The concentration of a density $\rho$ can be measured via the entropy functional $S$ defined by 
$$
{\rm S}(\rho) \, = \, -\int_{\R^{2n}} \rho \log\rho \,  d \mu_h .
$$
Note that this is the \textit{physical} entropy, which is the negative of the mathematical entropy.
Wehrl \cite{w79} conjectured $n$ to be a lower bound on the entropy
of phase space densities induced by $\mathcal{L}$ acting on $({\rm L}^2(\mathbb{R}^n), dx)$,
that is
$$
{\rm S}(\rho) \, \ge \,  n
$$
whenever $\rho=|\mathcal{L} \psi|^2$ and $\psi \in ({\rm L}^2(\R^n), dx)$.
Lieb \cite{li78} established this inequality with a method based on the sharp
Young and Haussdorf-Young inequalities. Carlen \cite{c91} recovered Lieb's result
via an approach based on the logarithmic Sobolev inequality and also settled the
problem of characterizing the cases of equality.    

In what follows we apply our results from the previous sections
to show that in some configurations,
one can obtain positive lower bounds on the {\it Wehrl deficit}
$$
\delta_{\rm Wehrl} (\rho) \, = \, {\rm S}(\rho) - n
$$
in terms of well-known metrics. The method of proof is based on Carlen's approach.

\begin{theorem} \label {thm.wehrl}
Suppose $\rho=|\mathcal{L} \psi|^2$ is a probability density on $(\R^{2n}, d\mu_h)$
with barycenter $b=b_\rho \in \R^{2n}$. Let 
$$
d\nu_\rho (z) \, = \, e^{\frac{|z|^2}{2}} \rho \bigg(\sqrt{\frac{h}{2\pi}} \, z\bigg) d\gamma(z),
$$ 
$$
d  \nu_{\rho,b} (z) 
     \, = \, e^{\frac{|z|^2}{2}} \rho \bigg(\sqrt{\frac{h}{2\pi}} \, z+b\bigg) d\gamma(z)
$$
where $\gamma $ is the standard Gaussian measure on $\R^{2n}$.
There exists $c>0$ such that if $\rho$ is not identically $e^{-\frac{\pi}{h}|z|^2}$,
then
$$ 
\delta_{\rm Wehrl} (\rho) \, \ge \,  \frac{c}{ {\rm H}(\nu_\rho)}
    \, \min \bigg ( \frac {{\rm W}_{1,1} ( \nu_{\rho,b}, \gamma ) ^4}{n^2} \, ,
       \frac {{\rm W}_{1,1} (\nu_{\rho,b}, \gamma )^2 }{ n} \bigg ).
$$ 
Moreover, in the class of probability densities $\rho$ with finite second moments,
$\delta_{\rm Wehrl}(\rho)=0$ exactly when $\rho = e^{-\frac{\pi}{h}|z-z_0|^2}$
for some $z_0 \in \mathbb{R}^{2n}$ or alternatively, when
$\psi_{p_0,q_0}(x)=e^{ip_0\cdot x} \phi_0(x-q_0)$ for some ${(p_0,q_0) \in \R^{2n}}$
and $\phi_0(x)=\big(\frac{2}{h}\big)^{\frac{n}{2}} e^{-\frac{|x|^2}{2}}$.   
\end{theorem}

\begin{proof}
Let $f_h$ be the density of $\nu_\rho$ with respect to $\gamma$ so that
$\int_{\R^{2n}} f_h  d\gamma = \int_{\R^{2n}} \rho \, d\mu_h=1$ and
$$
{\rm H} (\nu_\rho)
  \, = \, \int _{\R^{2n}} f_h \log f_h d\gamma 
   \,  = \,  \int_{\R^{2n}} \Big(\frac{\pi}{h} \rho \, |x|^2+\rho\log \rho \Big)d\mu_h.
 $$ 
Since $f_h$ is not identically $1$, the strict convexity of the
function $t \to t \log t$ implies (via Jensen) that ${\rm H}(\nu_\rho)>0$.
Since $\rho$ has finite first moment, ${\rm W}_{1,1} (\nu_{\rho,b}, \gamma ) < \infty$.
Thus, if ${\rm H}(\nu_\rho)=\infty$, there is nothing to prove, so we may assume
without loss that $\rho$ has finite second moments.  
     
A direct calculation shows that
$$ 
\frac{|\nabla f_h|^2}{f_h} \,=\,
     e^{\frac{|z|^2}{2}} \Bigg(\frac{h}{2\pi} \,
       \frac{\big |\nabla \rho \big (\sqrt{\frac{h}{2\pi}} \, z \big ) \big |^2}{\rho \big (\sqrt{\frac{h}{2\pi}}\big )}+2\sqrt{\frac{h}{2\pi}}  \, \nabla \rho \bigg (\sqrt{\frac{h}{2\pi}} \, z\bigg) \cdot z
           +\rho \bigg (\sqrt{\frac{h}{2\pi}} \, z \bigg)|z|^2 \Bigg )
 $$ 
and, by changing variables and using the divergence theorem,  
\begin{align*}
\int _{\R^{2n}}\frac{|\nabla f_h|^2}{f_h} d\gamma 
 & \, = \,  \int _{\R^{2n}} \bigg(\frac{h}{2\pi} \, \frac{|\nabla \rho|^2}{\rho}
     +2\nabla \rho  \cdot x + \frac{2\pi}{h} \, \rho \, |x|^2\bigg) d\mu_h\\
& \, = \, \int _{\R^{2n}} \bigg(\frac{h}{2\pi} \, \frac{|\nabla \rho|^2}{\rho}
     +\frac{2\pi}{h} \, \rho \, |x|^2\bigg) d\mu_h - 4n.
\end{align*}
Therefore, 
$$
\delta_{\rm LSI}(\nu_\rho) \, = \,  \frac{1}{2} \, {\rm I}(\nu_\rho) - {\rm H}(\nu_\rho)
  \, = \, \frac{h}{4\pi} \int_{\R^{2n}}  \frac{|\nabla \rho|^2}{\rho} \,d\mu_h + {\rm S}(\rho)-2n.
$$
Since $\rho=|\mathcal{L} \psi|^2$, an application of \cite[Theorem 6]{c91} yields 
$$
\int_{\R^{2n}}  \frac{|\nabla \rho|^2}{\rho} d\mu_h 
  \, = \,  4 \int_ {\R^{2n}} |\nabla \rho^{\frac{1}{2}}|^2 d\mu_h \, = \, \frac{4n\pi}{h} \, .
$$
Thus $\delta_{\rm LSI}(\nu_\rho) = {\rm S}(\rho)-n $ and Corollary \ref{cor.deficit1} implies  
$$      
{\rm S}(\rho) - n  \, \ge \,  \frac{c}{ {\rm H}(\nu_\rho)}
    \, \min \bigg ( \frac {{\rm W}_{1,1} ({\nu}_{\rho,b}, \gamma ) ^4}{4n^2} \, ,
       \frac {{\rm W}_{1,1} ({\nu}_{\rho,b}, \gamma )^2 }{ 2n} \bigg ) 
$$
where 
$$
  \nu_{\rho,b} (dz) \, = \,  f_h(z+b_h) e^{(-(b_h\cdot z+\frac{|b_h|^2}{2}))} d\gamma(z)
 \, = \, e^{|z|^2/2}\rho \bigg(\sqrt{\frac{h}{2\pi}}( z + b_h)\bigg) d\gamma(z)
$$
and $b_h $ is the barycenter of $f_h$ with respect to the Gaussian. To conclude the proof of the inequality, note that 
$$
b_h \, = \, \int_{\R^{2n}} z f_h d\gamma 
  \, = \,  \sqrt{\frac{2\pi}{h}} \int_{\R^{2n}} z \rho(z) d\mu_h  \, = \, \sqrt{\frac{2\pi}{h}} b_\rho.
$$
Next, assume that ${\rm S}(\rho)=n$. Since $\rho$ has finite second moments, ${\rm H}(\nu_\rho)<\infty$.
If ${\rm H}(\nu_\rho)=0$, Jensen's inequality ensures that $\rho$ has the desired form.
If $0<{\rm H}(\nu_\rho)<\infty$, it follows that $ \nu_{\rho , b} = \gamma$.
Thus, 
$$
e^{\frac{|z|^2}{2}} \rho \bigg(\sqrt{\frac{h}{2\pi}} \, z + b_\rho\bigg) \, = \,  1
$$
for some $b_\rho =(p_0,q_0) \in \R^{2n}$. Consequently,
$$
\rho(z) \, = \, e^{-\frac{\pi}{h}|z-b_\rho|^2} \, = \, |\mathcal{L} \Psi_{p_0,q_0}|^2
$$
and Lieb \cite{li78} has shown that the map $\psi \to |\mathcal{L} \psi|^2$ is injective.           
\end{proof}

In a similar way, one may use Corollaries~\ref{cor.stabilityw2}
and \ref{cor.tv} to obtain dimension-independent lower bounds
on the Wehrl deficit for a subclass of probability measures. For instance, Corollary~\ref{cor.stabilityw2} implies the following result.   

\begin{theorem} \label{thm.wehrl2}
Suppose $\rho=|\mathcal{L} \psi|^2$ is a probability density on $(\R^{2n}, d\mu_h)$
with barycenter $b = b_\rho \in \R^{2n}$, finite second moments, and
such 
$$
z \, \mapsto \,  e^{|z|^2/2} \rho \bigg(\sqrt{\frac{h}{2\pi}} \, z\bigg)
$$
satisfies a Poincar\'e inequality with constant $\lambda>0$.
Then
$$
\delta_{\rm Wehrl}(\rho) \, \ge \,  c_2(\lambda) \; {\rm W}_2( \nu_{\rho,b}, \gamma )^2,
$$
where $c_2(\lambda)$ is as in Corollary~\ref{cor.stabilityw2}. 
\end{theorem}

As an example of illustration, for $M>0$, let 
$$
\rho(z) \in \mathcal{F}_M \, = \,  \{e^{-\psi(z)}: {\rm Hess} (\psi) \ge M\}.
$$ 
Set
$$
f_h(z) \, = \, e^{|z|^2/2}\rho \bigg(\sqrt{\frac{h}{2\pi}} \, z\bigg)
$$ 
and note that 
$$
-{\rm Hess}\big ( \log(f_h) \big)
    \, = \,  \frac{h}{2\pi} \, {\rm Hess} (\psi) \bigg (\sqrt{\frac{h}{2\pi}} \, z \bigg) - {\rm Id}
   \,  \ge \,  \frac{Mh}{2\pi} - {\rm Id}.
$$ 
Thus, if $M>\frac{3\pi}{h}$, the previous theorem applies in $\mathcal{F}_M$.

\medskip

It is well known that the range of $\mathcal{L}$ is closely related to the space
${\cal A}^2$ of entire function $\Phi$ on $\mathcal{C}^n$ such that
$$
\int |\Phi(z)|^2 e^{-2\pi|z|^2/h} dp \, dq \, < \,  \infty
$$
where $z=(q+ip)/\sqrt{2}$. The precise statement is that for every
$\psi \in ({\rm L}^2(\R^n), dx)$, 
$$
\mathcal{L}\psi(p,q) \, = \, e^{ip\cdot q/2h^*} \Phi((q-ip)/\sqrt{2})e^{(p^2+q^2)/4h^*}
$$ where $\Phi \in {\cal A}^2$. In fact, Segal \cite{s62, s63} (see also \cite{s70})
proved that the map
$\mathcal{\widetilde L}: \psi \to \Phi$ is unitary from
$(L^2(\R^n), dx)$ onto ${\cal A}^2$, and therefore Carlen \cite{c91} calls
$\mathcal{\widetilde L}$ the Segal transform. With this in mind, the Segal transform may be useful in characterizing the subspace of functions $\psi$ in the domain of $\mathcal{L}$ mapping to functions
$|\mathcal{L} \Psi|^2$ admitting a Poincar\'e inequality and hence a
dimensionless $W_2$-estimate via Theorem~\ref {thm.wehrl2}.

\bigskip
\medskip

%{\red
%\noindent E. Indrei, Carnegy Mellon University (USA) egi@andrew.cmu.edu
%
%\noindent M. Fathi, Universit\'e Pierre et Marie Curie, Paris (France) max.fathi@etu.upmc.fr
%
%\noindent M. Ledoux, University of Toulouse, Toulouse (France) and Institut Universitaire de France\\
%ledoux@math.univ-toulouse.fr
%}

\end{document}